\documentclass[10pt, reqno]{amsart}

\usepackage{amsfonts}
\usepackage{amsthm}
\usepackage{mathrsfs}
\usepackage{amsmath}
\usepackage{epsfig}
\usepackage{url}

\theoremstyle{theorem}
\newtheorem{theorem}{Theorem}
\newtheorem{corollary}[theorem]{Corollary}
\newtheorem{lemma}[theorem]{Lemma}

\theoremstyle{definition}

\newtheorem*{remark}{Remark}

\DeclareMathOperator{\dist}{dist}

\newcommand{\M}{\mathbb{M}}
\newcommand{\C}{\mathbb{C}}
\newcommand{\Mydef}{\stackrel{\mbox{\footnotesize{\rm def}}}{=}}

\begin{document}
\title{A Resolvent Criterion for Normality}
\author{Cara D. Brooks}
\email{cbrooks@fgcu.edu}
\author{Alberto A. Condori}
\email{acondori@fgcu.edu}

\address{Department of Mathematics, Florida Gulf Coast University, Fort Myers, FL 33965}
\maketitle

\begin{abstract}
			Given a normal matrix $A$ and an arbitrary square matrix $B$ 
			(not necessarily of the same size), what relationships between 
			$A$ and $B$, if any, guarantee that $B$ is also a normal matrix?   
			We provide an answer to this question in terms of pseudospectra and 
			norm behavior.  In doing so, we prove that a certain distance formula, 
			known to be a necessary condition for normality, is in fact sufficient 
			and demonstrates that the spectrum of a matrix can be used to recover 
			the spectral norm of its resolvent precisely when the matrix is normal. 
			These results lead to new normality criteria and other interesting 
			consequences.  
\end{abstract}

\section{Normality, Pseudospectra and Norm Behavior.}

Let $n$ be a natural number and let $\M_{n}$ denote the set 
of all $n\times n$ matrices with entries in the complex plane $\C$.  
Denote by $I$ and $0$  the identity and zero matrices, respectively, 
whose sizes are understood in context.  

For $A\in\M_n$, let $A^{*}$ denote the conjugate transpose of $A$.  The 
spectrum $\sigma(A)$ of $A$ is the set of all of its eigenvalues; that is,
\[	\sigma(A)=\{z\in\C: zI-A\text{ is not invertible}\}
		=\{z\in\C: \det(zI-A)=0\}.	\]
		
In this article, we are interested in the notion of \emph{normality}.  Recall 
that a matrix $A\in\M_n$ is said to be \textbf{normal} if it commutes with 
its conjugate transpose, i.e., if
\[	A^{*}A=AA^{*}.	\]
In the case that $A^{*}A=AA^{*}=I$, the matrix $A$ is called 
\textbf{unitary}\footnote{For matrices with real entries, unitary matrices 
are called orthogonal.}.  Equivalently, $A$ is unitary if it is invertible 
and has $A^{*}$ as its inverse.   Two matrices $A$ and $B$ are said to be 
\textbf{unitarily similar} if there is a unitary matrix $U$ so that $A=UBU^{*}$.
  
The simplest example of an $n\times n$ normal matrix is a \textbf{diagonal} 
matrix, that is, a matrix of the form
\begin{equation}\label{diagWithLambdas}
			\Lambda=\left[
			\begin{array}{cccc}
						\lambda_1	&	0					&	\ldots	&	0\\
						0					&	\lambda_2	&	\ldots	&	0\\
						\vdots		&	\vdots		&	\ddots	&	\vdots\\
						0					&	0					&	\ldots	&	\lambda_n
			\end{array}\right].
\end{equation}
Notice that the eigenvalues of $\Lambda$ are precisely the entries on its
main diagonal and so $\sigma(\Lambda)=\{\lambda_1,\ldots,\lambda_n\}$.

Normal matrices are essentially diagonal according to the spectral 
theorem (for normal matrices) which states that a matrix $N\in \M_n$ is 
normal if and only $N$ is unitarily similar to a diagonal matrix. 
(See Theorem 2.5.4 on page 101 in \cite{HJ1}.)

So we ask, if given a normal matrix $A$ and an arbitrary square matrix $B$, 
what relationships between $A$ and $B$, if any, guarantee that $B$ is also 
a normal matrix?     

Certainly, by the Spectral Theorem stated above,  a ``non-metric'' relationship 
that guarantees the normality of $B$ from that of $A$ is unitary similarity.  
This however requires that $A$ and $B$ also have the same size.  

And so instead we ask, 
\begin{quote}
			\emph{Given a normal matrix $A$ and an arbitrary square matrix $B$ (not 
			necessarily of the same size), what ``metric'' relationships between 
			$A$ and $B$, if any, guarantee that B is also a normal matrix?}
\end{quote}

Two such relationships that come to mind are ``identical pseudospectra''
and ``same norm behavior.'' To describe these notions, we first need to 
choose a norm for matrices. 

For $A\in\M_n$, we define the (spectral) norm $\|A\|$ of $A$ by
\[	\|A\|=\sup\{\|Av\|_2: \|v\|_2=1\},	\]
where $\|v\|_2$ denotes the Euclidean norm of the vector $v\in\C^n$, i.e.,
\[	\|v\|_2=\sqrt{|v_1|^2+\cdots+|v_n|^2}\;\text{ if }v=(v_1,\ldots,v_n).	\]  
(The subscript allows one to differentiate between the norms.)  A useful 
feature of the matrix norm chosen here is \textbf{unitary invariance}, 
that is, 
\begin{equation}\label{unitInv}
			\|VTU\|=\|T\|
\end{equation}
for any $T\in\M_n$ and $n\times n$ unitary matrices $U$ and $V$.

Two square matrices $A$ and $B$ (\emph{not} necessarily of the same size) 
have \textbf{identical pseudospectra}\footnote{In view of inequality 
\eqref{distIneq} below, we adopt the convention that $\|(zI-A)^{-1}\|=\infty$ 
for $z\in\sigma(A)$.} if
\begin{equation}\label{pseudospectraAB}
			\|(zI-A)^{-1}\|=\|(zI-B)^{-1}\|\quad\text{ for all }z\in\C.
\end{equation} 

In the case of a \emph{normal} matrix $A\in\M_n$, the spectral theorem
allows one to compute the norm of the \textbf{resolvent} $(zI-A)^{-1}$ 
of $A$.  By that theorem, there is a diagonal matrix $\Lambda$ and a 
unitary matrix $U$ so that $A=U\Lambda U^{*}$ and so 
$zI-A=U(zI-\Lambda)U^{*}$ for any $z\in\C$.  It follows that 
\[	\|(zI-A)^{-1}\|=\|(zI-\Lambda)^{-1}\|
		=\max\{|z-\lambda|^{-1}:\lambda\in\sigma(A)\}	\]
or equivalently\footnote{If $z\in\C$ and $E\subseteq\C$, 
we define $\dist(z,E)=\inf\{|z-w|:\,w\in E\}$.},
\begin{equation}\label{resolventFormula}
		\|(zI-A)^{-1}\|=\frac{1}{\dist(z,\sigma(A))}\;
		\text{ for }z\notin\sigma(A).
\end{equation}

As the following theorem confirms, the condition of identical 
pseudospectra guarantees that $B$ is normal whenever $A$ is.

\begin{theorem}\label{pseudopCondThm}
			Suppose $A$ and $B$ are square matrices\footnote{In 
			Theorem \ref{pseudopCondThm}, the matrices $A$ and $B$ 
			are \emph{not} assumed to have the same size.}.
			If $A$ is normal, and $A$ and $B$ have identical pseudospectra,
			then $B$ is normal.
\end{theorem}

To prove Theorem \ref{pseudopCondThm}, it is convenient to know whether 
the validity of the distance formula \eqref{resolventFormula} for an 
``arbitrary'' matrix $B$ implies that $B$ is normal. This turns out to be true.

\begin{theorem}\label{normalIff}
			For a matrix $T \in \M_{n}$ to be normal, it is necessary and 
			sufficient that
			\begin{equation}\label{distFormula}
						\|(zI-T)^{-1}\|=\frac{1}{\dist(z,\sigma(T))}\;
		\text{ for }z\notin\sigma(T).
			\end{equation}
\end{theorem}

The necessity of \eqref{distFormula} was addressed in our remarks prior 
to the statement of Theorem \ref{pseudopCondThm} and is well known, e.g.,
see problem 6.42 on page 62 in \cite{K}.  In Section \ref{distSection} 
below, we establish the sufficiency of \eqref{distFormula} and some of 
its consequences.

Thus, in the remainder of this section, we prove Theorem \ref{pseudopCondThm} 
(assuming the validity of Theorem \ref{normalIff}) and discuss its 
consequences;  Theorem \ref{normalIff} is proved in the next section.

\begin{proof}[Proof of Theorem \ref{pseudopCondThm}]
			Suppose that $A$ is normal, and that $A$ and $B$ have identical 
			pseudospectra. Not only does \eqref{pseudospectraAB} imply that $A$ 
			and $B$ have the same spectrum $\sigma$, it also implies that
			\[	\|(zI-B)^{-1}\|=\frac{1}{\dist(z,\sigma)}
					\quad\text{ for all }z\notin\sigma	\]
			by \eqref{resolventFormula} because $A$ is normal.  Thus, the 
			normality of $B$ follows from Theorem \ref{normalIff}.
\end{proof}

Recall that if $p(z)=c_0+c_1 z+c_2 z^2+\cdots+c_m z^m$ is a polynomial
with complex coefficients and $T\in\M_n$, then $p(T)$ denotes the $n\times n$ 
matrix defined by
\[	p(T)=c_0 I+c_1 T+c_2 T^2+\cdots+c_m T^m.	\] 
Using this definition, it can be verified that, for fixed $T$, the 
mapping $p\mapsto p(T)$ is both linear and multiplicative.

Two square matrices $A$ and $B$ (\emph{not} necessarily of the same 
size) have the \textbf{same norm behavior} if 
\begin{equation}\label{normbehavior}
			\|p(A)\|=\|p(B)\| \quad \text{ for all polynomials } p.
\end{equation}

It is worth mentioning that the norm of a polynomial and the norm 
of the resolvent of a matrix appear naturally in applications.  For 
instance, the stability of a linear dynamical system is determined
by the norm $\|p(A)\|$ for suitable polynomials $p$.  For example, given 
a discrete system $v_{k+1}=Av_k$, $k\geq 0$, we have $v_{k}=A^{k}v_0$ for 
all $k\geq 0$ and so it suffices to use 
$p(z)=z^{k}$.  Likewise, given a continuous system $y^{\prime}=Ay$,
there is a polynomial $p_{t}$ with coefficients depending on $t$ so that
$p_{t}(A)=\exp(tA)$ (see \cite{Pu}) and so the behavior of $y$ is determined 
by $\|p_{t}(A)\|$ because $y=\exp(tA)y(0)$.  Furthermore,  if $\|(zI-A)^{-1}\|$ 
is known, one can obtain (not necessarily sharp) upper bounds for $\|p(A)\|$ 
using the Cauchy integral formula
\[	p(A)=\frac{1}{2\pi i}\int_{\Gamma}(zI-A)^{-1}p(z)\,dz,	\]
where $\Gamma$ is any contour enclosing $\sigma(A)$ (see page 46 in \cite{K}). 

In \cite{GT}, Greenbaum and Trefethen showed that if two matrices have 
the same norm behavior, then they also have identical pseudospectra.
Therefore, an application of their result and Theorem \ref{pseudopCondThm} 
give

\begin{corollary}\label{sameNormBehaviorCor}
			Suppose $A$ and $B$ are square matrices\footnote{In Corollary 
			\ref{sameNormBehaviorCor}, the matrices $A$ and $B$ are \emph{not} 
			assumed to have the same size.}.  If $A$ is normal, and 
			$A$ and $B$ have the same norm behavior, then $B$ is normal.
\end{corollary}

In the case that both matrices $A$ and $B$ have the same size, 
one obtains criteria for their unitary similarity.  This is stated in the 
following corollary.

\begin{corollary}\label{GCorollary}
			Let $A,B\in\M_n$ and suppose $A$ is normal.  The following statements 
			are equivalent.
			\begin{enumerate}
						\item\label{similarity}
									$A$ and $B$ are unitarily similar.
						\item\label{normBehaviorPlusChi}
									$A$ and $B$ have the same norm behavior and 
									characteristic polynomials.
						\item\label{pseudoPlusChi}
									$A$ and $B$ have identical pseudospectra and 
									characteristic polynomials.
			\end{enumerate}
\end{corollary}
\begin{proof}
			Condition 1 and the normality of $A$ imply that 
			both $A$ and $B$ are unitarily similar to the same diagonal matrix, 
			say $D$.  In particular, if $A=UDU^{*}$ for some unitary matrix $U$, 
			then $p(A)=Up(D)U^{*}$.  Therefore, $\|p(A)\|=\|p(D)\|$ for any 
			polynomial $p$ by \eqref{unitInv}, and likewise $\|p(B)\|=\|p(D)\|$.
			Moreover, if $A=VBV^{*}$ for some unitary $V$, then factoring
			$zI-A=V(zI-B)V^{*}$ implies that
			\[	\det(zI-A)=\det(zI-B)	\]
			because the determinant is a multiplicative map.  Hence, condition
			\ref{normBehaviorPlusChi} holds.
			
			The fact that condition \ref{normBehaviorPlusChi} implies condition
			\ref{pseudoPlusChi} is an immediate consequence of the result 
			mentioned from \cite{GT}.
			
			Finally, if condition \ref{pseudoPlusChi} holds, Theorem 
			\ref{pseudopCondThm} implies that both $A$ and $B$ are 
			normal matrices of the same size.  In addition, they also
			have the same eigenvalues (counting multiplicities) because
			they have the same characteristic polynomials.  In other words,
			by the spectral theorem, $A$ and $B$ are unitarily similar to 
			the same diagonal matrix, and so condition \ref{similarity} 
			holds, as desired.
\end{proof}

Note that the implication 
``\ref{normBehaviorPlusChi}$\implies$\ref{similarity}'' 
in Corollary \ref{GCorollary} is also a consequence of Corollary 
\ref{sameNormBehaviorCor} above; indeed, that corollary implies 
that $A$ and $B$ are normal matrices of the same size and so unitary 
similarity follows again by the spectral theorem.  For yet another 
proof of the implication ``\ref{normBehaviorPlusChi}$\implies$\ref{similarity},''
we refer the reader to \cite{G}.

\section{The distance formula, its consequences, and criteria for normality.}\label{distSection}

We now give a straightforward and self-contained proof of Theorem \ref{normalIff} and 
state some of its own interesting consequences.  We also mention two other 
approaches by which Theorem \ref{normalIff} can be proved. 

\begin{proof}[Proof of Theorem \ref{normalIff}]
			In view of the remarks preceding \eqref{resolventFormula}, 
			we only need to show that if \eqref{distFormula} holds, 
			then $T$ is normal. Instead of \eqref{distFormula}, let us 
			assume the following equivalent formulation:
			\begin{equation}\label{maxFormula}
						\|(zI-T)^{-1}\|
						=\max\left\{|z-\lambda|^{-1}:\lambda\in\sigma(T)\right\}
						\quad\text{ for all }z\notin\sigma(T).
			\end{equation}
			By Schur's Theorem (see Theorem 2.3.1 on page 79 and the Remark 
			on page 80 in \cite{HJ1}), there is a unitary matrix $U$ and a lower
			triangular matrix $L$ such that $T=ULU^{*}$.  In this case, the 
			main diagonal of $L$ must consist of the eigenvalues of $T$
			(in any desired order but counting multiplicities).    
			So, $zI-L$ is lower triangular and has $(z-\lambda)$ with 
			$\lambda\in\sigma(T)$ as main diagonal entries.  Consequently, 
			$(zI-L)^{-1}$ is also lower triangular with entries $(z-\lambda)^{-1}$,
			$\lambda\in\sigma(T)$, on its main diagonal. Since 
			$(zI-T)^{-1}=U(zI-L)^{-1}U^{*}$, \eqref{maxFormula} can be 
			restated as
			\begin{equation}\label{maxFormulaL}
						\|(zI-L)^{-1}\|
						=\max\left\{|z-\lambda|^{-1}:\lambda\in\sigma(L)\right\}
						\quad\text{ for all }z\notin\sigma(L).
			\end{equation}
			We show that $L$ must be a diagonal matrix.  To that end, we use
			the following well-known result (see Section 0.7.3 in \cite{HJ1}).
			
			\begin{lemma}\label{invLemma}
						Suppose $M\in\M_n$ is invertible and has block form
						\begin{equation}\label{blockM}
									M=\left[
									\begin{array}{cc}
												A	&	0\\
												B	&	C
									\end{array}\right],
						\end{equation}
						where $A$ and $C$ are square matrices.  Then $A$ and $C$ 
						are invertible, and the inverse of $M$ has block form
						\begin{equation}\label{blockInvM}
									\left[
									\begin{array}{cc}
												A^{-1}					&	0\\
												-C^{-1}BA^{-1}	&	C^{-1}
									\end{array}\right].
						\end{equation}
			\end{lemma}			
			
			To simplify notation, label the entries down the main diagonal in $L$ 
			as $\lambda_n$, $\lambda_{n-1}$, $\ldots$, $\lambda_2$, and $\lambda_1$. 
			For $k=1,\ldots,n$, let $L_k$ denote the principal submatrix of $L$ 
			obtained by removing the first $n-k$ rows and columns of $L$.  In 
			particular, each $L_k$ is a $k\times k$ lower triangular matrix, 
			$L_{n}=L$, and $L_1=\lambda_1$.  
			
			For each $1\leq k\leq n-1$, 
			\[	L_{k+1}=\left[
					\begin{array}{cc}
								\lambda_{k+1}	&	0\\
								b_{k}					&	L_{k}
					\end{array}\right]	\]
			for some vector $b_k\in\C^{k}$ and so, by Lemma \ref{invLemma}, 
			\begin{equation}\label{blockRes}
						(zI-L_{k+1})^{-1}=\left[
						\begin{array}{cc}
									(z-\lambda_{k+1})^{-1}	&	0\\
									(zI-L_{k})^{-1}b_{k}(z-\lambda_{k+1})^{-1}
									&	(zI-L_{k})^{-1}
						\end{array}\right]
			\end{equation}
			for every $z \notin \sigma(L)$.  In particular, if 
			$v=(1,0,\ldots,0)\in\C^{k+1}$, then
			\[	\|(zI-L_{k+1})^{-1}\|^{2}\geq\|(zI-L_{k+1})^{-1}v\|_2^{2},	\]
			or equivalently,
			\begin{equation}\label{opEstimate}
						\|(zI-L_{k+1})^{-1}\|^{2}\geq
						|z-\lambda_{k+1}|^{-2}
						+\|(zI-L_{k})^{-1}b_{k}(z-\lambda_{k+1})^{-1}\|_{2}^{2}.
			\end{equation}
			Similarly, we see that
			\begin{equation}\label{smallerRes}
						\|(zI-L_{n})^{-1}\|
						\geq\|(zI-L_{n-1})^{-1}\|
						\geq\ldots
						\geq\|(zI-L_{1})^{-1}\|.
			\end{equation}
			We now show that $b_k=0$. Indeed, if $z\in\C$ is chosen so that 
			the maximum on the right-hand side of \eqref{maxFormulaL} equals 
			$|z-\lambda_{k+1}|^{-1}$, then
			\begin{align*}
						|z-\lambda_{k+1}|^{-2}&=\|(zI-L_{n})^{-1}\|^{2}\\
						&\geq\|(zI-L_{k+1})^{-1}\|^{2}\\
						&\geq|z-\lambda_{k+1}|^{-2}
						+\|(zI-L_{k})^{-1}b_{k}(z-\lambda_{k+1})^{-1}\|_{2}^{2}
			\end{align*}
			by \eqref{maxFormulaL} and \eqref{opEstimate}. Therefore,
			\[	(zI-L_{k})^{-1}b_{k}(z-\lambda_{k+1})^{-1}=0	\]
			and so $b_{k}=0$.  
			Since each column below a main diagonal entry of $L$ 
			is a zero vector, $L$ is a diagonal matrix and so
			$T$ is normal, as desired.
\end{proof}

Even though the distance formula \eqref{distFormula} does 
not hold for \emph{non-normal} $n\times n$ matrices, the 
following inequality does: for any $T\in\M_n$,
\begin{equation}\label{distIneq}
			\dist(z,\sigma(T))\geq\|(zI-T)^{-1}\|^{-1}
			\quad\text{ for all }z\notin\sigma(T). 
\end{equation}
This fact is verified using the Neumann (geometric) series;    
for if $A$ and  $B$ are in $\M_n$, $A$ is invertible, and 
$\|A-B\|<\|A^{-1}\|^{-1}$ holds, then $BA^{-1}$ is invertible 
and so  $B$ is also invertible.  It follows that if $(zI-T)$ 
is invertible and the inequality  
$|z-w|=\|(zI-T)-(wI-T)\|<\|(zI-T)^{-1}\|^{-1}$
holds, then $(wI-T)$ is also invertible.  In other words, if 
$z\notin\sigma(T)$ and $|z-w|<\|(zI-T)^{-1}\|^{-1}$, then
$w\notin\sigma(T)$ and so the inequality in \eqref{distIneq} 
is obtained.

Our proof of Theorem \ref{normalIff} also reveals that equality 
\eqref{distFormula} need not hold \emph{for all} 
$z\notin\sigma(T)=\{\lambda_1,\ldots,\lambda_n\}$.  Rather, normality 
of $T$ follows from the weaker (equivalent) condition\footnote{Even 
better, if $d$ is the number of 
distinct eigenvalues of $T$ and one eigenvalue is simple, then equality 
\eqref{pointseq} need only hold for $z_k$ near each of the other $d-1$ 
eigenvalues to conclude that $T$ be normal.} that for 
each $1\leq k\leq n-1$,	there is a $z_k\in\C$ so that
\[	\|(z_k I-T)^{-1}\|=|z_k-\lambda_k|^{-1}.	\]

We now state the following criteria for normality.

\begin{theorem}\label{completethm}
			The following statements are equivalent for a matrix $T\in\M_n$.
			\begin{enumerate}
						\item\label{distFormCond}
									For all $z\notin\sigma(T)$, 
									\[	\|(zI-T)^{-1}\|=\frac{1}{\dist(z,\sigma(T))}.	\]
						\item	For each $1\leq k\leq n-1$, there is a $z_k\in\C$ so that									
									\begin{equation}\label{pointseq}
												\|(z_k I-T)^{-1}\|=|z_k-\lambda_k|^{-1},
									\end{equation}
									where $\lambda_1,\ldots,\lambda_n$ denote the eigenvalues
									of $T$ (counting multiplicities).
						\item\label{normalCond}	$T$ is normal.
						\item\label{pnormCond}
									For every polynomial $p$,
									\begin{equation}\label{pNormCondition}
												\|p(T)\|=\max\{|p(\lambda)|:\lambda\in\sigma(T)\}.
									\end{equation}
			\end{enumerate}
\end{theorem}
\begin{proof}
			It is only left to prove that ``\ref{pnormCond}$\implies$\ref{distFormCond}'' as
			 ``\ref{normalCond}$\implies$\ref{pnormCond}'' is a straightforward consequence 
			of the spectral theorem; for if $T$ is normal and $p$ is a polynomial, 
			there is a unitary matrix $U$ so that $T=U\Lambda U^{*}$ with $\Lambda$ 
			as in \eqref{diagWithLambdas}.  In this case,
			\[	\|p(T)\|=\|Up(\Lambda)U^{*}\|=
					\|p(\Lambda)\|=\max\{|p(\lambda)|:\lambda\in\sigma(T)\}	\]
			and so condition \ref{pnormCond} holds.
			
			Now suppose condition \ref{pnormCond} holds.
			Let $z\notin\sigma(T)$ be fixed and define $f(t) =(z-t)^{-1}$.  
			Let $\lambda_1, \ldots, \lambda_k$ denote the distinct 
			eigenvalues of $T$.  For $1\leq i\leq k$, $\lambda_i$ is a zero of the 
			minimal polynomial $m$ of $T$.  It follows from Theorem 6.2.9 in 
			\cite{HJ2} that $(zI-T)^{-1} = q_z(T)$, where $q_z(t)$ is any polynomial 
			of degree at most $n-1$ that interpolates $f(t)$ and 
			its derivatives at the zeros of $m$. That is, $f^{(u)}(\lambda_i) 
			= q_z^{(u)}(\lambda_i)$ for $u=0,\ldots,(s_i-1)$
			and $i=1, \ldots, k$ (see page 390 in \cite{HJ2}).      
			Therefore, \eqref{pNormCondition} implies
			$\|q_{z}(T)\|=\max\{|q_{z}(\lambda)|:\lambda\in\sigma(T)\}$,
			or equivalently,
			\[	\|(zI-T)^{-1}\|=\max\{|z-\lambda|^{-1}:\lambda\in\sigma(T)\}
					=\frac{1}{\dist(z,\sigma(T))}
					\quad\text{ for }z\notin\sigma(T).	\]
			Hence, condition \ref{distFormCond} holds and the proof is now complete.
\end{proof}

Surprisingly, the criteria for normality appearing in Theorem \ref{completethm}
are absent from the literature.  Thus, these criteria may be considered addenda 
to the 89 other characterizations of normal matrices that appear in \cite{EI}
and \cite{GJSW}.  

\begin{corollary}\label{twoCorollary}
			If $T\in\M_2$ and 
			\[	\|(zI-T)^{-1}\|=\frac{1}{\dist(z,\sigma(T))}	\]
			holds for one point $z\notin\sigma(T)$, then $T$ is normal.
\end{corollary}

Hence, for non-normal $2\times 2$ matrices $T$, the inequality in 
\eqref{distIneq} must be \emph{strict} for all $z\notin\sigma(T)$.  
For instance, if
\[	N=\left[
		\begin{array}{cc}
					0	&	1\\
					0	&	0
		\end{array}\right],		\]
then $\sigma(N)=\{0\}$ and so 
$\|(zI-N)^{-1}\|>\dist^{-1}(z,\sigma(N))=|z|^{-1}$ holds for 
all $z\neq 0$ by \eqref{distIneq} and Corollary \ref{twoCorollary}.

\begin{remark} The sufficient condition for normality in 
			Theorem \ref{normalIff} may also be deduced two other ways. 
			\begin{enumerate}
			\item If $\sigma_{\epsilon}(T)$ denotes the 
			``$\epsilon$-pseudospectrum'' of $T$, that is,
			\[	\sigma_{\epsilon}(T)\Mydef\{z:\|(zI-T)^{-1}\|>\epsilon^{-1}\},	\] 
			then Theorem 2.2 of \cite{TE} states that\footnote{The set on 
			the right-hand side of \eqref{TEcondition} can be rewritten as 
			$\{z:\dist(z,\sigma(T))<\epsilon\}$.} if the equality
			\begin{equation}\label{TEcondition}
						\sigma_{\epsilon}(T)
						=\{\zeta+\xi: \zeta\in\sigma(T)\text{ and }|\xi|<\epsilon \}
			\end{equation} 
			holds for all $\epsilon>0$, then the matrix $T$ is normal. 
			In a sketch of a proof, the authors assert that one can deduce 
			simultaneous diagonalizability of $T$ and $T^*$ from content in 
			a later section in \cite{TE} on eigenvalue perturbation theory.  
			Therefore, if equality \eqref{distFormula} holds, then the 
			equality of the sets in \eqref{TEcondition} follows, and so $T$ 
			is normal. 
			
			\item The fact that \eqref{distFormula} implies normality of $T$ 
			may also be deduced in the context of radial matrices.  
			For $z \notin \sigma(T)$, equality \eqref{distFormula} 
			implies that $(zI-T)^{-1}$ is radial and so unitarily similar to
			\begin{equation}\label{Horneq}
						 \|(zI-T)^{-1}\|\left(U \oplus B\right),
			\end{equation}
			where $U\in\M_k$ is unitary, $1\leq k\leq n$, and $B\in\M_{n-k}$
			has spectral radius less than 1 and $\|B\|\leq 1$
			(see Problem 27(g) on page 45 in \cite{HJ2}).  To conclude normality 
			of $T$, one can obtain from \eqref{Horneq} orthonormal eigenvectors
			of $(zI-T)^{-1}$ (corresponding to the eigenvalue $(z-\lambda)^{-1}$)  
			and so orthonormal eigenvectors for $T$; thus, one can form a complete 
			set of orthonormal eigenvectors of $T$ through a choice of $z$ near each 
			eigenvalue. 
			\end{enumerate}
\end{remark}

\textbf{Acknowledgments.}
The authors wish to thank Thomas Ransford and Roger Horn for comments 
on a preliminary version of this paper calling to our attention 
Theorem 2.2 of \cite{TE} and an alternative approach to obtaining 
Theorem \ref{normalIff} using unitary similarity to \eqref{Horneq}, 
respectively.


\begin{thebibliography}{XXX}
			\bibitem{EI}		L. Elsner and Kh. D. Ikramov.
											Normal matrices: an update. 
											\emph{Linear Algebra Appl.} \textbf{285} (1998), no. 1-3, 291--303.
			\bibitem{G}			T. G. Gerasimova. 
											Unitary similarity to a normal matrix.
											\emph{Linear Algebra Appl.} \textbf{436} (2012), no. 9, 3777--3783.
			\bibitem{GT}		A. Greenbaum and L. N. Trefethen.
											Do the pseudospectra of a matrix determine its behavior?
											Technical Report TR 93-1371, Department of Computer Science,
											Cornell University, 1993. 
			\bibitem{GJSW}	R. Grone, C. R. Johnson, E. M. Sa, and H. Wolkowicz.
											Normal matrices.  
											\emph{Linear Algebra Appl.} \textbf{87} (1987), 213--225.
			\bibitem{HJ1}		R. A. Horn and C. R. Johnson.
											\emph{Matrix Analysis.}
											Cambridge Univ. Press, Cambridge, 1985.
			\bibitem{HJ2}		R. A. Horn and C. R. Johnson.
											\emph{Topics in Matrix Analysis.}
											Cambridge Univ. Press, Cambridge, 1991.
			\bibitem{K}			T. Kato.  
											\emph{A Short Introduction to Perturbation Theory for Linear Operators.}
											Springer-Verlag, New York-Berlin, 1982.
			\bibitem{Pu}		E. J. Putzer.
											Avoiding the Jordan canonical form in the discussion of linear systems
											with constant coefficients.  
											\emph{Amer. Math. Monthly} \textbf{73} (1966), 2--7.
			\bibitem{TE}  	L. N. Trefethen and M. Embree. 
											\emph{Spectra and Pseudospectra: The Behavior of Nonnormal 
											Matrices and Operators.}  Princeton Univ. Press, 
											New Jersey, 2005.
\end{thebibliography}
\end{document}